\documentclass[11pt,reqno]{article}
\usepackage{amsmath,amsthm,amsfonts,amssymb,amscd}
\usepackage[latin1]{inputenc}

\usepackage{amssymb}
\usepackage{amsthm}
\usepackage{graphicx}
\usepackage{amscd}

\newtheorem{Theorem}{Theorem}
\newtheorem{Corollary}{Corollary}
\newtheorem{Proposition}{Proposition}
\newtheorem{Lemma}{Lemma}

\newtheorem{Claim}{Claim}
\theoremstyle{Definition}
\newtheorem{Definition}{Definition}

\theoremstyle{Remark}
\newtheorem{Remark}{Remark}

\def\leaderfill{\leaders\hbox to .8em{\hss .\hss}\hfill}
\def\_#1{{\lower 0.7ex\hbox{}}_{#1}}
\def\esima{${}^{\text{\b a}}$}
\def\esimo{${}^{\text{\b o}}$}
\font\bigbf=cmbx10 scaled \magstep1

\def\A{{\mathcal{A}}}
\def\sa{{\mathcal{S}}}
\def\P{{\mathcal{P}}}
\def\C{{\mathcal{C}}}
\def\L{{\mathcal{L}}}
\def\G{{\mathcal{G}}}

\def\fa{{\mathcal{F}}}
\def\O{{\mathcal{O}}}
\def\eR{{\mathcal{R}}}
\def\M{{\mathcal{M}}}
\def\D{{\mathcal{D}}}
\def\U{{\mathcal{U}}}
\def\B{{\mathcal{B}}}
\def\E{{\mathcal{E}}}
\def\H{{\mathcal{H}}}
\def\po{{\partial}}
\def\ro{{\rho}}
\def\te{{\theta}}
\def\Te{{\Theta}}
\def\om{{\omega}}
\def\Om{{\Omega}}
\def\vr{{\varphi}}
\def\ga{{\gamma}}
\def\Ga{{\Gamma}}
\def\la{{\lambda}}
\def\La{{\Lambda}}
\def\ov{\overline}
\def\al{{\alpha}}
\def\ve{{\varepsilon}}
\def\lg{{\langle}}
\def\rg{{\rangle}}
\def\lv{{\left\vert}}
\def\rv{{\right\vert}}
\def\be{{\beta}}

\def\bh{{\mathbb{H}}}
\def\bp{{\mathbb{P}}}
\def\ee{{\mathbb{E}}}
\def\re{{\mathbb{R}}}
\def\bz{{\mathbb{Z}}}
\def\bq{{\mathbb{Q}}}
\def\bd{{\mathbb{D}}}
\def\bc{{\mathbb{C}}}
\def\bn{{\mathbb{N}}}
\def\bk{{\mathbb{K}}}

\def\Sep{\operatorname{{Sep}}}
\def\Re{\operatorname{{Re}}}
\def\Mon{\operatorname{{Mon}}}
\def\SL{\operatorname{{SL}}}
\def\Res{\operatorname{{Res}}}
\def\Fol{\operatorname{{Fol}}}
\def\tr{\operatorname{{tr}}}
\def\dim{\operatorname{{dim}}}
\def\Aut{\operatorname{{Aut}}}
\def\GL{\operatorname{{GL}}}
\def\Aff{\operatorname{{Aff}}}
\def\Hol{\operatorname{{Hol}}}
\def\loc{\operatorname{{loc}}}
\def\Inv{\operatorname{{Inv}}}
\def\Ker{\operatorname{{Ker}}}
\def\mI{\operatorname{{Im}}}
\def\Dom{\operatorname{{Dom}}}
\def\Id{\operatorname{{Id}}}
\def\Tni{\operatorname{{Int}}}
\def\supp{\operatorname{{supp}}}
\def\Diff{\operatorname{{Diff}}}
\def\sing{\operatorname{{sing}}}
\def\Sing{\operatorname{{sing}}}
\def\codim{\operatorname{{codim}}}
\def\grad{\operatorname{{grad}}}
\def\Ind{\operatorname{{Ind}}}
\def\deg{\operatorname{{deg}}}
\def\rank{\operatorname{{rank}}}
\def\sep{\operatorname{{Sep}}}
\def\Sat{\operatorname{{Sat}}}
\def\signal{\operatorname{{signal}}}

\title{Extension theorems for analytic objects associated to  foliations}

\author{César Camacho and Bruno Sc\'ardua}

\date{}

\begin{document}

\maketitle

\begin{abstract}
In this paper we will establish a structure theorem concerning the
extension of analytic objects associated to germs of dimension one
foliations on surfaces, through one-dimensional barriers. As an application, an
extension theorem for projective transverse structures is obtained.

\end{abstract}


\def\leaderfill{\leaders\hbox to .8em{\hss .\hss}\hfill}
\def\_#1{{\lower 0.7ex\hbox{}}_{#1}}
\def\esima{${}^{\text{\b a}}$}
\def\esimo{${}^{\text{\b o}}$}
\font\bigbf=cmbx10 scaled \magstep1

\def\A{{\mathcal{A}}}
\def\sa{{\mathcal{S}}}
\def\P{{\mathcal{P}}}
\def\C{{\mathcal{C}}}
\def\L{{\mathcal{L}}}
\def\G{{\mathcal{G}}}
\def\fa{{\mathcal{F}}}
\def\O{{\mathcal{O}}}
\def\eR{{\mathcal{R}}}
\def\M{{\mathcal{M}}}
\def\D{{\mathcal{D}}}
\def\U{{\mathcal{U}}}
\def\B{{\mathcal{B}}}
\def\E{{\mathcal{E}}}
\def\T{{\mathcal{T}}}

\def\po{{\partial}}
\def\ro{{\rho}}
\def\te{{\theta}}
\def\Te{{\Theta}}
\def\om{{\omega}}
\def\Om{{\Omega}}
\def\vr{{\varphi}}
\def\ga{{\gamma}}
\def\Ga{{\Gamma}}
\def\la{{\lambda}}
\def\La{{\Lambda}}
\def\ov{\overline}
\def\al{{\alpha}}
\def\ve{{\varepsilon}}
\def\lg{{\langle}}
\def\rg{{\rangle}}
\def\lv{{\left\vert}}
\def\rv{{\right\vert}}
\def\be{{\beta}}

\def\bh{{\mathbb{H}}}
\def\bp{{\mathbb{P}}}
\def\ee{{\mathbb{E}}}
\def\re{{\mathbb{R}}}
\def\bz{{\mathbb{Z}}}
\def\bq{{\mathbb{Q}}}
\def\bd{{\mathbb{D}}}
\def\bc{{\mathbb{C}}}
\def\bn{{\mathbb{N}}}
\def\bk{{\mathbb{K}}}

\def\Re{\operatorname{{Re}}}
\def\SL{\operatorname{{SL}}}
\def\Res{\operatorname{{Res}}}
\def\Fol{\operatorname{{Fol}}}
\def\tr{\operatorname{{tr}}}
\def\dim{\operatorname{{dim}}}
\def\Aut{\operatorname{{Aut}}}
\def\GL{\operatorname{{GL}}}
\def\Aff{\operatorname{{Aff}}}
\def\Hol{\operatorname{{Hol}}}
\def\loc{\operatorname{{loc}}}
\def\Ker{\operatorname{{Ker}}}
\def\mI{\operatorname{{Im}}}
\def\Dom{\operatorname{{Dom}}}
\def\Id{\operatorname{{Id}}}
\def\Tni{\operatorname{{Int}}}
\def\supp{\operatorname{{supp}}}
\def\Diff{\operatorname{{Diff}}}
\def\sing{\operatorname{{sing}}}
\def\Sing{\operatorname{{sing}}}
\def\codim{\operatorname{{codim}}}
\def\grad{\operatorname{{grad}}}
\def\Ind{\operatorname{{Ind}}}
\def\deg{\operatorname{{deg}}}
\def\rank{\operatorname{{rank}}}
\def\sep{\operatorname{{sep}}}
\def\Sat{\operatorname{{Sat}}}
\def\sep{\operatorname{{sep}}}
\def\signal{\operatorname{{signal}}}
\def\ord{\operatorname{{ord}}}

\section{Introduction}

A regular one-dimensional foliation on a complex surface is given by
an atlas of distinguished neighborhoods $\{U_j\}$, $j\in J$,
covering the manifold, and for each $j\in J$ by a submersion $y_j
\colon U_j \to \bc$ defining the foliation, such that on each
nonempty intersection $U_i \cap U_j \ne \emptyset$ we have $
dy_i=g_{ij}\, dy_j$ where $g_{ij}\in\mathcal O^*(U_i\cap \textsc{U}_j)$
is a not vanishing holomorphic function defined on $U_i\cap U_j$. A
complex one dimensional foliation with isolated singularities on a
complex surface $M$ is a regular foliation of $M\setminus S$, where
$S$ is a discrete set of points of $M$. Each element of $S$ is
called an isolated $\it{singularity}$ of the foliation. An
elementary application of Hartog's extension theorem (\cite{Gunning})
shows that in the
neighborhood of each singularity the foliation can be defined by a
holomorphic one-form. We assume that the one-form vanishes at the
singularity, otherwise the foliation would have a regular extension.
Thus a foliation with a discrete set of singularities on a complex
manifold $M$ can be defined by an atlas $\{U_j\}$, $j\in J$,
covering $M$ and for each $j\in J$ a holomorphic one-form $\omega_j$
defining the foliation on $U_j$, such that on each nonempty
intersection $U_i \cap U_j \ne \emptyset$ we have $\large
\omega_i=g_{ij}.\omega_j$ where $g_{ij}\in\mathcal O^*(V_i\cap
\textsc{U}_j)$ is a not vanishing holomorphic function defined on
$U_i\cap U_j\ne \emptyset$. Whenever the set $S$ has cardinality
greater than one, we say that we are dealing with a global
foliation. A simple example of a global foliation is obtained by
blowing-up an isolated singular point $ 0\in \bc^2$ of a foliation
$\fa$ defined in a neighborhood $0\in U\subset \bc^2$ by a
holomorphic one-form $\omega$, vanishing only at $ 0\in\bc^2$. Let
$(x,y)$ be coordinates of $\bc^2$ restricted to $U$. Define a
complex 2-manifold $\mathcal U$ by glueing two charts defined by the
coordinates: $ U_1=(x,t)$, $ U_2=(u,y)$ such that $(x,y)\in U$
$u,t\in \bc$, $y=t.x$, $u.t=1$. The map $\pi_0: \mathcal U\to U$
defined on these charts by $\pi_0(x,t)=(x,tx)$, $\pi_0(u,y)=(uy,y)$
is a proper holomorphic map, $D_0=\pi_0^{-1}(0)$ is the exceptional
divisor, isomorphic to an embedded projective line, and $\pi_0:
\mathcal U\setminus D_0 \to U\setminus \{0\}$ is a biholomorphism.
On these charts $\pi_0^*(\omega)=x^\nu.\omega_1$,
$\pi_0^*(\omega)=y^\nu.\omega_2$, where $\nu$ is a positive integer,
depending on the algebraic multiplicity of the singularity, and
$\omega_1$, $\omega_2$ are holomorphic 1-forms with isolated
singularities. Then, the 1-forms $\omega_1$, $\omega_2$ satisfy
$\omega_1=g_{12}.\omega_2$, $g_{12}\in \mathcal O^*(U_1\cap U_2)$
and define a foliation $\fa_0$ on $\mathcal U$ called the ${\it
analytic}$ $\it{ extension}$ of $\pi^*\fa$ on $\mathcal U\setminus
D$ to $\mathcal U$.

We have two possibilities. Either $D_0$ is tangent to $\fa_0$, {\it i.e.}, $D_0$
is a leaf plus a finite number of singularities, and in this case
we say that $D_0$ is $\it nondicritical$, or $D_0$ is transverse
to $\fa_0$ everywhere except at a finite number of points that can be either
singularities or tangency points of $\fa_0$ with $D_0$. In this last case we say
that $D_0$ is $\it dicritical$.

This process can be repeated at each one of the singularities, or
tangency points of $\fa_0$ with $D_0$. Seidenberg \cite{Seidenberg}
states that by composition of a finite number of these blow-up's we
can obtain a proper holomorphic map $\pi: \tilde U \to U$ such that
$\pi^{-1}(0)= \cup_{j=0}^m D_j$ is a finite union of embedded
projective lines with normal crossings, called the {\it exceptional
divisor}. This map is called the {\it resolution morphism} of $\fa$.
 Any component $D_j$ is either invariant or everywhere
transverse to the pull back foliation $\tilde\fa= \pi^*(\fa)$. Any
singular point of $\tilde\fa$ will be {\it irreducible} in the following
sense.

  Let $\omega= a(x,y) dx+ b(x,y) dy$ be a holomorphic one-form defined
in a neighborhood of $0\in \bc^2$. We say that $0\in \bc^2$ is a
$\it singular$ point of $\omega$ if $a(0,0)=b(0,0)=0$, and  a $\it
regular$ point otherwise. The vector field $X=(-b(x,y),a(x,y))$ is
in the kernel of $\omega$. The nonsingular orbits of $X$ are the
leaves of the foliation.

We say that $0\in \bc^2$ is an
$\it{irreducible}$ singular point of
$\omega$ if the eigenvalues
$\lambda_1, \lambda_2$ of the linear
part of $X$ at $0\in \bc^2$ satisfy one of the following conditions:

(1) $\lambda_1.\lambda_2\neq 0$ and $\lambda_1/\lambda_2\notin
{\bq_+}$

(2) either $\lambda_1\neq 0$ and $\lambda_2= 0$, or viceversa.

In case (1) there are two invariant curves tangent to the
eigenvectors corresponding to
$\lambda_1$ and $\lambda_2$. In case (2) there is
an invariant curve tangent at $0\in\bc^2$ to the eigenspace
corresponding to $\lambda_1$. These curves are called
$\it{separatrices}$ of the foliation.

Suppose that $0\in \bc^2$ is either a regular point or an irreducible
singularity of a foliation $\mathcal I$.
It is possible to show that in suitable local coordinates $(x,y)$ in a
neighborhood
$0\in U \in \bc^2$
of the origin, we have the following local normal forms for the one-forms
defining this foliation:
\begin{itemize}

\item [ (Reg)] $dx=0$, whenever $0\in \bc^2$ is a regular point of $\mathcal I$.

and if $0\in\bc^2$ is an irreducible singularity of $\tilde\fa$,
then either

\item[(Irr.1)] $xdy - \la ydx + \eta_2(x,y) = 0$ where $\la \in
\bc\backslash\bq_+$\,, \, $\eta_2(x,y)$ is a holomorphic one-form
with a zero of order $\ge 2$ at $(0,0)$. This is called {\it
nondegenerate singularity}. Such a  singularity is {\it resonant} if
$\lambda\in \mathbb Q_-$ and {\it hyperbolic} if $\lambda \notin
\mathbb R$, or

\item[(Irr.2)]
 $y^{t+1}dx - [x(1+\la y^t) + A(x,y)]dy = 0$\, , where
$\la \in \bc$, \, $t \in \bn = \{1,2,3,\dots\}$ and $A(x,y)$ is a
holomorphic function with a zero  of order $\ge t+2$ at $(0,0)$.
This is called {\it saddle-node singularity}. The {\it strong
separatrix} of the saddle-node is given by  $\{y=0\}$. If the
singularity admits another separatrix then it is necessarily smooth
and transverse to the strong separatrix, it can be taken as the other
coordinate axis and will be called {\it central} manifold of the
saddle-node.
\end{itemize}

 In the last two cases we have
$\{y=0\} \subset \sep(\mathcal I,U) \subset \{xy=0\}$, where $\sep(\mathcal I,U)$
denotes the union of separatrices of $\mathcal I$ through $0\in\bc^2$.

 A {\it fundamental domain} of $\mathcal I|_U$ at $0\in \bc^2$ is a
subset $\mathcal D\subset \bc^2$ consisting of:

\begin{itemize}
\item[{\rm(E.0)}] In the regular case: a neighborhood of
$0\in\bc^2$ minus a codimension one submanifold passing through
$0\in\bc^2$, transverse to the foliation.

\item[{\rm(E.1)}] A neighborhood of the singularity minus its separatrices in case the singularity is
nondegenerate nonresonant.

\item[{\rm(E.2)}] A neighborhood of the singularity minus its separatrices,  union a neighborhood
of an annulus around $0\in \bc^2$ contained in one of the separatrices, in case the
singularity is resonant.

\item[{\rm(E.3)}] A neighborhood of the singularity minus its separatrices,  union a neighborhood
of an annulus around $0\in \bc^2$  contained in the strong separatrix,  in case the
singularity is a saddle-node.

\end{itemize}

Conditions (E.2) and (E.3) are related to the fact that, in the resonant
case and in the saddle-node case, the holonomy of the mentioned
separatrix characterizes the analytical type of the foliation (cf.
\cite{Martinet-Ramis1}, \cite{Martinet-Ramis2}).

\subsection{The Globalization Theorem}

Consider now an arbitrary germ of an analytic foliation $\fa$ at an
isolated singularity $0\in \bc^2$ with resolution morphism $\pi\colon \tilde U \to U$.
A {\it separatrix} of $\fa$ at $0\in
\bc^2$ is the germ at $0\in \bc^2$ of an irreducible analytic curve
which is invariant by $\fa$. It follows from the resolution theorem
that for a small enough neighborhood $0\in U\subset\bc^2$ any
separatrix of $\fa$ can be represented in $U$ by an irreducible analytic
curve passing through $0\in\bc^2$ which is the closure of a leaf of
$\fa|_U$. We will write $\sep(\fa,U)$ to denote the union of these
separatrices. By Newton-Puiseux parametrization theorem, if $U$ is
small enough, there is an analytic injective map $f \colon \mathbb D
\to U$ from the unit disk $\mathbb D \subset \mathbb C$ onto the
separatrix, mapping the origin to $0\in \bc^2$, and nonsingular
outside the origin $0 \in \mathbb D$. Therefore a separatrix locally
has the topology of a punctured disk. We shall say that the
separatrix is {\it resonant} if for any loop in the punctured disk
that represents a generator of the homotopy of the leaf, the
corresponding holonomy map is a resonant diffeomorphism. Choose a
holomorphic vector field $X$ which generates the foliation $\fa|_U$,
and has an isolated singularity at $0\in \bc^2$. Then, the
separatrix is called {\it resonant} if the loop $\gamma$ generating the homotopy
of the leaf in the separatrix satisfies $\exp \int_\gamma \tr(DX)$
is a root of the unity.

The main concept we introduce is the following:

 \begin{Definition}
 {\rm A {\it fundamental domain} of $\fa|_U$ at $0\in\bc^2$ is a
subset $\mathcal D\subset U$ which is :

\begin{itemize}

\item[{\rm(i)}]  A fundamental domain of $\fa|_U$ at $0\in\bc^2$  whenever
$0\in\bc^2$ is either a regular point or an irreducible singularity,
and

\item[{\rm(ii)}] In case $0\in\bc^2$ is a not irreducible singularity, a subset
$\mathcal D\subset\bc^2$ written as $\mathcal D= (U\setminus
\sep(\fa,U))\cup \mathcal S$, where $\mathcal S\subset U$ is the
union of ring neighborhoods of loops $\gamma$, one for each resonant
separatrix.
\end{itemize}
}\end{Definition}

It is important to remark that the pull-back of a fundamental domain
by the resolution morphism, is a fundamental domain for some
singularities of $\tilde \fa$, but not necessarily for all of them.
This is the case, for instance, for saddle-nodes with strong manifold
tangent to the resolution divisor.

  Let $U$ be a neighborhood of $0\in\bc^2$, as above. A meromorphic $q$-form $\xi$
  defined on $U\setminus \sep(\fa,U)$ is called
  {\it extensible with respect  to} $\fa|_U$ in $U$
  if any extension of $\xi$ to a fundamental domain of $\fa|_U$
  extends as a meromorphic $q$-form to $U$.
  We will say also that $\xi$ is {\it extensible with respect to the germ} $\fa$ at
 $0\in\bc^2$ if it is extensible with respect to $\fa|_U$ in some neighborhood
  $0\in U\subset \bc^2$.

In general it is a not trivial task to prove that a $q$-form is
extensible with respect  to a local foliation. We show in section 4
that one-forms associated to projective transverse structures of a
foliation $\mathcal I$ are extensible with respect to $\mathcal I$.

  Let $U$ be a neighborhood of $0\in\bc^2$, as above. A meromorphic $q$-form $\xi$
  defined on $U\setminus \sep(\fa,U)$ is called {\it  infinitesimally extensible with
  respect  to} $\fa|_U$ in $U$ if $\tilde\xi:=\pi^*\xi$
 is  an extensible $q$-form with respect to $\tilde\fa|_{\pi^{-1}(U)}$ at a generic point
 on each dicritical component and in a neighborhood of each irreducible singular point of $\tilde\fa$.

A natural question is to find extension theorems for general germs
of foliations. We will show next that for any germ of a foliation it
is enough to check extensibility at the irreducible singularities
produced in the process of desingularization.

\begin {Theorem}[Globalization theorem]
  \label{Theorem:Globalization}
 Let $\fa$ be the germ of a holomorphic foliation with an isolated
 singularity at $0\in \bc^2$. For a small enough neighborhood
 $0\in U\subset \bc^2$
 any meromorphic $q$-form infinitesimally extensible with respect to
 $\fa$ in $U$ is extensible.
 \end{Theorem}

 \section{Resolution of singularities}

\subsection{The Index theorem}
Let $\sigma$ be a Riemann surface embedded in a two dimensional
manifold $S$ ; $\fa$ a foliation on $S$ which leaves $\sigma$
invariant and $q\in \sigma$. There is a neighborhood of $q$ where
$\sigma$ can be expressed by $(f=0)$ and $\fa$ is induced by the
holomorphic 1-form $\omega$ written as
$\omega=hdf+f\eta$.
Then we can associate the following index:
$$
i_{q}(\fa, \sigma):=-{\rm Residue}_q(\frac{\eta}{h})|_{\sigma}
$$
relative to the invariant submanifold $\sigma$. A nondegenerate singularity
in the form (Irr.1) has two
invariant manifolds crossing normally, they correspond to the $x$
and $y$-axes. In this case if $\sigma$ is locally $(y=0)$ and $q=0$, this index is equal
to $\lambda$
(quotient of eigenvalues). The saddle-node in (Irr.2) has an invariant manifold
corresponding to the $x$-axis and, depending on the higher order terms,
it may not have another invariant curve (see \cite{Martinet-Ramis2}).
In the case of a saddle-node, if
$ \sigma $ is equal to $(x=0)$ and $q=0$, this index is $\lambda$, and
if $\sigma$ is equal to $(y=0)$ and $q=0$, this index is zero.
At a regular point $q$ of $\fa$ the index is zero.
The index theorem of \cite{Camacho-Sad} asserts that the sum of
all the indices at the points in $\sigma$ is equal to the
self-intersection number $\sigma\cdot\sigma$:
$$
\sum_{q\in{\sigma}} i_q(\fa,\sigma)=\sigma\cdot\sigma.
$$
\subsection{Resolution of singularities: linear chains}

Suppose $\fa$ is a complex one-dimensional foliation defined
on an open neighborhood $0\in U\subset \bc^2$.
The resolution process of $\fa$ at $0\in \bc^2$ can be described and
ordered as follows.
The blow-up of $\fa$ at $0\in \bc^2$ is $(U_0, \pi_0, D_0, \fa_0)$
where $\pi_0: U_0\to U$ is the usual blow-up map ( see § 1). Then,
$U_0$ is a complex 2-manifold,
$D_0=\pi_0^{-1}(0)\subset U_0$
is an embedded projective line called the
$\it{exceptional}$ $\it{  divisor}$,
and the restriction of the map $\pi_0$ to $U_0\setminus D_0$
is a biholomorphism from
$ U_0\setminus D_0$ to $U\setminus \{0\}$.
Moreover $\fa_0$ is the analytic foliation on $U_0$ obtained by extension
to $D_0$ of $(\pi_0|_{U_0\setminus D})^*\fa$, as defined in the Introduction.
We also observe that the Chern class of the normal bundle to
$D_0\subset U_0$ is $-1$.
We have two possibilities. Either $D_0$ is tangent to $\fa_0$, i.e. $D_0$
is a leaf plus a finite number of singularities, and in this case
we say that $D_0$ is $\it nondicritical$, or $D_0$ is transverse
to $\fa_0$ everywhere except at a finite number of points that can be either
singularities or tangency points of $\fa_0$ with $D_0$. In this last case we say
that $D_0$ is $\it dicritical$.

Proceeding by induction we define the
step $\underbar 0$ as the first blow-up $(U_0, \pi_0, D_0, \fa_0)$.
We assume that $(U_k, \pi_k, D_k, \fa_k)$ has been already defined, where
$\pi_k: U_k\to U$ is a holomorphic map, such that $D_k=\pi_k^{-1}(0)$ is
a divisor, union of a finite number of embedded projective lines with
normal crossing. The crossing points of $D_k$ are called $\it corners$.
The restriction of $\pi_k$ to $ U_k\setminus D_k$ is
a biholomorphism from $U_k\setminus D_k$ to $U\setminus \{0\}$. The
foliation $\fa_k$ on $U_k$ is the analytic extension to $D_k$ of the
foliation $(\pi_k|_{U_k\setminus D_k})^*\fa$.

Let $p_0: \tilde U_k \to U_k$ be the blow-up at a point $r\in D_k$, outside the
corners.
Let $P_0= p_0^{-1}(r)$ be the exceptional divisor. We write
$\tilde D_k=\ov{ p_0^{-1}(D_k\setminus \{r\})}$ and
$\tilde r= P_0\cap \tilde D_k$.
If $P$ is the irreducible component of $D_k$ containing $r$ we will denote by
$\tilde P = \ov{p_0^{-1}(P\setminus\{r\})}$. Then it is easy to see (\cite{Camacho-Sad}) that
$i_{\tilde r}(\tilde P)= i_r(P) - 1$.
 Using the fact that the restriction of $p_0$ to $\tilde U_k\setminus P_0$
 is a biholomorphism onto $U_k\setminus \{r\}$ we will say that $r$ becomes
 $\tilde r$ after one blow-up and also simplify notations
 identifying  $\tilde D_k$ with $D_k$, $\tilde P$
with $P$ and $\tilde r$ with $r$. Thus in the new notation,
$(\pi_k\circ p_0)^{-1}(0)= D_k\cup P_0$ and we will say that
$r= P_0\cap D_k$ was blown-up once.

We proceed to define $(U_{k+1}, \pi_{k+1}, D_{k+1}, \fa_{k+1})$ as follows.
Let $\tau_k\subset D_k$ be the set of points outside the corners
of $D_k$, that are either tangency points of $\fa_k$
with $D_k$ or not irreducible singular points of $\fa_k$. Let $r\in\tau_k$.
We introduce at $r$ a $\it linear$ $\it chain$ $\mathcal C(r)$ {\it with origin at}
$r\in D_k$, by
means of a sequence of blow-up's, first at the point $r$, the precise number of
times necessary to become either irreducible, or regular and then at any
reducible corner produced in this way. The resolution theorem of
Seidenberg \cite{Seidenberg} guarantees that after a finite number of blow-up´s
all corners obtained in this process will be either irreducible singular
points or regular points.

The linear chain $\mathcal C(r)$ can be seen as an ordered finite
sequence of embedded projective lines: $P_m> P_{m-1}> ...> P_1$
where $r=D_k\cap P_m$ and if $i>j$ and $P_i\cap P_j\neq\emptyset$
then $i=j+1$ and $P_{j}\cap P_{j+1}$ is just one point. For any
$l=1,..., m-1$ write $r_l=P_l\cap P_{l+1}$. Two invariants can be
associated to $\mathcal C(r)$. One is the {\it order} $n_r$ of $\mathcal C(r)$
defined as the the minimun number of times that was necessary to blow-up
$r$ in order to become
irreducible, and the $\it {length}$  $m$  of the linear chain. Given
any number $ 1\leq t < m$ we will say that the sequence
$P_t> P_{t-1}> ...> P_1$
is a linear chain $\mathcal C(r_t)$ of length $t$ with origin at
$r_t\in P_t\cap P_{t+1}$. We will write $|\mathcal C(r)|=\cup _{j=1}^m P_j$
to denote the $\it{support}$ of the chain $\mathcal C(r)$.

Let $-k_l=P_l.P_l$ be the self intersection number of $P_l$ in the
linear chain  $\mathcal C(r)$. The sequence of numbers $n_r.k_m....k_1$
belongs to the collection $\mathcal A$ of numbers defined as follows.
Start with $1.1\in\mathcal A$ and assume that $a_0.a_t.a_{t-1}....a_1$
belongs to $\mathcal A$. Then $(a_0+1.1.(a_t+1).a_{t-1}....a_1\in \mathcal A$,
and $a_0.a_t...(a_{j+1}+1).1.(a_j+1)....a_1\in\mathcal A$.

\begin{Lemma}[\cite{Camacho-Sad}]

Suppose that $a_0.a_t.a_{t-1}....a_1\in\mathcal A$. Then

$$
a_0=[a_t,a_{t-1},...,a_2,a_1]:=\frac{1}{a_t - \frac{1}{a_{t-1} -
\frac{1}{_{\ddots \, \, \, \, \frac{1}{a_2 - \frac{1}{a_1}}}}}}
$$

\end{Lemma}

We also have the following

\begin{Lemma}[\cite{Camacho-Sad}]
If $a_0.a_t...a_1\in \mathcal A$, then

a) $[a_l,..., a_h]>0$ if $1\leq h\leq l\leq t$ and $t\geq 2$

b) $0<[a_t,..., a_{t-i}]<[a_t,..., a_1]$ for $0\leq i\leq t-2$
\end{Lemma}

Let $p_1, p_2,..., p_u$ be the ordered sequence of blow-up's that created the
linear chain $\mathcal C(r)$, then the composition $p=p_u \circ...\circ p_2 \circ p_1$,
is a map $p:\tilde U(r)\to U_k$ for which
$p^{-1}(r)=|\mathcal C(r)|=\cup _{l=1}^m P_l$, where each $P_l$ is an embedded
projective line and $r_l=P_l\cap P_{l+1}$ is just a point, and
$r_m= P_m\cap D_k$ where we are making the identification
$D_k\equiv\ov{p^{-1}(D_k\setminus \{r\})}$, using the fact that the restriction of $p$ to
$\tilde U(r)\setminus |\mathcal C(r)|$ is a biholomorphism onto
$U_k\setminus \{r\}$.

Repeating this process at each of the points of $\tau_k$ we obtain,
by composition of these maps, a holomorphic map
$p_{k+1}: U_{k+1}\to U_k$
such that $p^{-1}_{k+1}(\tau_k)=\cup_{r\in \tau_k}|\mathcal C(r)|$,
a union of the supports of the linear chains with origin at the points in $\tau_k$.
Moreover
$p_{k+1}: U_{k+1}\setminus \cup_{r\in \tau_k}|\mathcal C(r)|\to U_k\setminus \tau_k$
is a biholomorphism.
Define $D_{k+1}:=D_k\cup_{r\in\tau_k}|\mathcal \C(r)|$ where we have
identified $D_k$ with $\ov{ p_{k+1}^{-1}(D_k\setminus \tau_k)}$.
Finally, we define
$\pi_{k+1}: U_{k+1}\to U$ by $\pi_{k+1}:= \pi_k \circ p_{k+1} $, and $\fa_{k+1}$
as the analytic extension of $(p_{k+1}|_{U_{k+1}\setminus D_k})^*\fa_k$ to $D_{k+1}$.

The theorem of Seidenberg asserts that this process
ends after a finite number of steps. We observe that the dicritical
components in the final configuration are disjoint, have no
singularities and are everywhere transverse to the foliation. The
resolution of $\fa$ at $0\in\bc^2$ is $(U_{n}, \pi_{n}, D_{n},
\fa_{n})$ if all the singularities of $\fa_n$ in $D_n$ are
irreducible but at least one singularity of $\fa_{n-1}$ in $D_{n-1}$
is not irreducible.
\section{Proof of the Globalization theorem}

Let $U$ be a neighborhood of $0\in\bc^2$, small enough such that any
separatrix of $\fa_U$ is an irreducible curve, union of a leaf of
$\fa_U$ and the point $0\in\bc^2$. Let $\xi$ be a meromorphic
$q$-form defined on a fundamental domain $\mathcal D=V\setminus
\sep(\fa, V)\cup \mathcal S$ of $\fa|_V$ at $0\in \bc^2$, where
$\mathcal S$ is the union of ring neighborhoods of generating
cycles, one for each separatrix. Consider a generic linear chain
created at the k-step in the process of resolution with origin at a
point $r\in P\subset D_k$, $\mathcal C(r)=(P_l)_{l=1}^m$, where $P$
is the irreducible component of $D_k$ containing $r$. As before
denote by $p:U(r)\to U_k$ the sequence of blow-up's that defined
$\mathcal C(r)$. Then $ p_k\circ p: \tilde U(r)\to U$ defines
$\tilde\fa(r)=(p_k\circ p)^*\fa|_V= p^*(p_k^*(\fa))=p^*(\fa_k)$ in
the neighborhood $\tilde U(r)$ of $|\mathcal C(r)|\cup D_k$. Define
also
 $\tilde{\mathcal D}= (p_k\circ p)^{-1}(\mathcal D)$. We will write
  $P=P_{m+1}$, $\tilde U=\tilde U_r$, $\tilde\fa$=$\tilde\fa(r)$,
  $\tilde{\mathcal D}=\tilde{\mathcal D}(r)$ and $r=r_{m+1}$, for
simplicity.

Then $\tilde{\mathcal D}\subset \tilde U$ can be written as $\tilde
{\mathcal D}= \tilde U\setminus {\sep(\tilde\fa, \tilde U)}\cup
\tilde {\mathcal S}$, where $\tilde{\mathcal S}$ is the union of
ring neighborhoods of generating loops $\gamma$, one for each
resonant separatrix not contained in $|\mathcal C(r)|\cup D_k$ and
  $\sep(\tilde\fa, \tilde U)$ is the union of $|\mathcal C(r)|$ and the
  separatrices of $\tilde\fa$. By hypothesis the meromorphic $q$-form
  $\tilde\xi$, defined on $\tilde{\mathcal D} $ is
extensible with respect to $\tilde\fa$ on $\tilde U$ at each
singularity of $\tilde\fa$ and at a generic point in each dicritical
component of $\mathcal C(r)$. However, $\tilde{\mathcal D}$ is not
necessarily a fundamental domain for each singularity of
$\tilde\fa$.

Denote by $-k_l$, where $k_l$ is a positive integer, the
self-intersection number of $P_l\subset U_k$. Let $\zeta_l\subset
P_l$ be the set of singular points of  $\tilde\fa$ in $P_l$ which
are not corners of $\mathcal C(r)$ and have a positive index
relative to $P_l$. Clearly $c_l:=\sum_{p\in \zeta_l} i_p(P_l)$ is a
not negative number and we define $\bar k_l=k_l+c_l$, for
$l=1,...,m$. Define also $\zeta=\cup_{l=1}^m \zeta_l$, and assume
that any singularity in $\cup_{l=1}^m P_l$ outside $\zeta$ is
irreducible.

We will say that $\tilde\xi$ has a meromorphic extension to a
neighborhood of  $|\mathcal C(r)|\setminus \zeta$ if for any compact
subset $K\subset {|\mathcal C(r)|\setminus\zeta}$ there is a
neighborhood $K\subset \tilde U_K\subset \tilde U$ in $\tilde U$ and
a meromorphic extension of $\tilde\xi$ to $\tilde U_K$.

\begin{Definition}[minimal chain]
{\rm A linear chain $\mathcal C(r)=( P_l)_{l=1}^m $ is called
$\it{minimal}$ if any corner $r_{l}=P_{l}\cap P_{l+1}$, $l=1,..., m$ is
of one of the following types:

(i) a regular point.

(ii) a saddle-node singularity with $i_{r_l}(P_{l+1})=0$

(iii) a resonant singularity with
$i_{r_{l}}(P_l)+ i_{r_{l-1}}(P_l)= -\bar k_l$, if $l>1$, and
$i_{r_1}(P_1)=-\bar k_1$ for $l=1$.
}
\end{Definition}

\begin{Proposition}
\label{Proposition:locallyextensible} Suppose that $\mathcal
C(r)=(P_l)_{l=1}^m $ is a linear chain of $\tilde\fa$ containing no
dicritical components such that any singularity in $|\mathcal
C(r)|\setminus \zeta$ is irreducible. Assume there is a meromorphic
$q$-form $\tilde\xi$ defined on $\tilde {\mathcal D}$. Then, either
$\tilde \xi$ has a meromorphic extension to a neighborhood of
$|\mathcal C(r)|\setminus \zeta$, or $\mathcal C(r)$ is a minimal
chain.
\end{Proposition}

\begin{proof}

We proceed by induction on the length $m$ of the linear chain
$\mathcal C(r)$. Suppose that $m=1$. Then $r=P_1\cap P_2$. Assume
first that $r$ is a nondegenerate nonresonant singularity, then as
$\tilde\xi$ is extensible with respect  to $\tilde\fa$ at $r$ it
extends as a meromorphic $q$-form to a neighborhood of $r$. By
Levi's extension theorem (\cite{Siu}) there is an arbitrarily small
neighborhood $N_1$ of the separatrices transverse to $P_1$ different
than $P_2$ such that $\tilde\xi$ extends as a meromorphic $q$-form
to $\tilde U\setminus N_1$.

Given an irreducible singular point $p\in P_1\setminus \zeta_1$,
then there are two possibilities. Either $i_p(\tilde\fa,P_1)\neq 0$,
and so there is a separatrix $\tilde s$ of $\tilde\fa$ at $p$,
transverse to $P_1$. If $p$ is either a resonant or a saddle-node
then the holonomy of $\tilde s$ is resonant and so $\tilde\xi$ is
defined in a fundamental domain at $p$, consequently it extends as a
meromorphic $q$-form to a neighborhood of $p$. Similarly, if $p$ is
not degenerate, and since the $q$-form $\tilde\xi$ is extensible
with respect  to $\tilde\fa$, at $p$, then it extends as a
meromorphic $q$-form to a neighborhood of $p$. On the other hand, if
 $i_p(\tilde\fa,P_1)= 0$
then $p$ is a saddle-node with its strong invariant manifold
contained in $P_1$. Moreover $\tilde\xi$ is defined in a fundamental
domain at $p$ and since it is extensible with respect to $\tilde\fa$
at $p$, it can be extended  to a neighborhood of $p$ as a meromorphic one-form.
Thus, we can assume that $r$ is either a saddle-node or a resonant
singularity. Suppose that $\mathcal C(r)$ is not minimal. Then
$i_{r}(P_2)\neq 0$ $\it and$ $i_{r}(P_2)\neq -1/\bar k_1$, and $r$
is either a saddle-node singularity with index $i_{r}(P_1)= 0$ or a
resonant singularity with $i_{r}(P_1)\neq - \bar k_1$. In both cases
we have that $i_{r}(P_1)+ c_1\neq-k_1$, and so by the index theorem
there is a singular point $p\in P_1\setminus \{r,\zeta_1\}$ with
$i_p(P_1)\neq 0$ not positive. By hypothesis $p$ is irreducible,
then there is a separatrix $\tilde s$ of $\tilde\fa$ at $p$,
transverse to $P_1$. If $p$ is either a resonant or a saddle-node
then the holonomy of $\tilde s$ is resonant and so $\tilde\xi$ is
defined in a fundamental domain at $p$ and so it extends as a
meromorphic $q$-form to a neighborhood of $p$. Similarly, if $p$ is
not degenerate the $q$-form $\tilde\xi$, extensible with respect to
$\tilde\fa$ at $p$, extends as a meromorphic $q$-form to a
neighborhood of $p$. Therefore, by Levi´s extension theorem
(\cite{Siu}) we can extend $\tilde\xi$ as close to $r$ as we wish.
Since $r$ is either a saddle-node with the strong separatrix tangent
to $P_1$ or a resonant singularity, then $\tilde\xi$ is already
defined in a fundamental domain at $r$ and so it can be extended to
a neighborhood of $r$. From this we obtain that $\tilde \xi$ extends
to a neighborhood of $P_1\setminus \zeta_1$.

Fix any integer $t$, $2\leq t\leq m$, and assume that the
alternative stated in the theorem holds true for linear chains of
length $t-1$. Then we have two possibilities:  either a) $\tilde\xi$
has been extended to $P_1\cup...\cup P_{t-1}\setminus
\zeta_1\cup...\cup \zeta_{t-1}$, or b) the linear chain $\mathcal
C(r_{t-1})$ is minimal. Consider the linear chain $\mathcal
C(r_{t})$ of  length $t$ and assume that $i_{r_t}(P_{t+1})\neq 0$.
If $r_t$ is a not degenerate, nonresonant singularity, then
$\tilde\xi$ extends to a neighborhood of $r_t$ and from there to a
neighborhood of $P_1\cup...\cup P_t\setminus \zeta_1\cup...\cup
\zeta_t$. We can then assume that $r_t$ is either a saddle-node with
$i_{r_t}(P_{t})= 0$ or a resonant singularity. If case a) happens
then $\tilde\xi$ is well defined in a neighborhood of $r_{t-1}$,
then by Levi's theorem $\tilde\xi$ will extend as close to $r_t$ as
desired. Then $\tilde\xi$ is defined in a fundamental domain at
$r_t$ and therefore extends as a meromorphic $q$-form to a
neighborhood of $r_t$ and thus to $P_1\cup...\cup P_t\setminus
\zeta_1\cup...\cup \zeta_t$. In case b) we have that either
$i_{r_{t-1}}(P_t)=0$, or $i_{r_{t-1}}(P_t)=-[\bar k_{t-1},..., \bar
k_h]$, where $h$ is the greatest positive integer $2\leq h \leq
{t-1}$ such that $i_{r_{h-1}}(P_h)=0$. It is easy to see from Lemma
2, that $-[\bar k_{t-1},..., \bar k_h]> -\bar k_t$. Thus $\bar k_t +
i_{r_{t-1}}(P_t) >0$. Suppose further that $i_{r_t}(P_{t+1})\neq 0$
and $i_{r_t}(P_{t+1})\neq -1/\bar k_t+i_{r_{t-1}}(P_t)$, then either
$r_t$ is a saddle-node with $i_{r_t}(P_t)=0$ or it is a resonant
singularity with $i_{r_t}(P_t)\neq -\bar k_t- i_{r_{t-1}}(P_t)$. In
any case $i_{r_t}(P_t)+ i_{r_{t-1}}(P_t)\neq -\bar k_t$ and
therefore there exists $p\in P_t\setminus \{r_{t-1},{r_t},\zeta_t\}$
such that $i_p(P_t)\neq 0$. Thus we can extend $\xi$ through $p$ to
a neighborhood of $r_{t-1}$ and $r_t$ and then to a neighborhood of
$P_1\cup...\cup P_t\setminus \zeta_1\cup...\cup \zeta_t$. It is
clear that the only alternative left is $ i_{r_t}(P_{t+1})=0$ or
$i_{r_t}(P_{t+1})= -1/\bar k_t+i_{r_{t-1}}(P_t)$. This last equation
is equivalent to $i_{r_t}(P_t)+ i_{r_{t-1}}(P_t)=-\bar k_t$.

\end{proof}

 \begin{Lemma}
 Suppose that in the linear chain $\mathcal C(r_m)$, $P_{m+1}$ is dicritical and the $P_l$,
 $l=1,..., m$ are nondicritical. Then $\tilde \xi$ can be extended
 to a neighborhood of
 $P_1\cup...\cup P_m\setminus \zeta$.
 \end{Lemma}

 \begin{proof}
 If $\tilde\xi$ extends to a neighborhood of
 $P_1\cup...\cup P_{m-1}\setminus \zeta$, then in particular it is well defined in a
 neighborhood of $r_{m-1}$. Thus it can be extended to a neighborhood of
 $r_m$. Suppose on the other hand that the linear chain $\mathcal C(r_{m-1})$ is minimal.
 Then either
 $i_{r_{m-1}}(P_m)=0$,
 or
 $i_{r_{m-1}}(P_m)=-[\bar k_{m-1},..., \bar k_h]$,
 where $h$ is the greatest positive integer
 $2\leq h \leq {m-1}$ such that $i_{r_{h-1}}(P_h)=0$. Since $[\bar k_{m},..., \bar k_h]>0$
 we have that $\bar k_m >[\bar k_{m-1},..., \bar k_h]$, and so $i_{r_{m-1}}(P_m)>-\bar k_m$.
 This is the same as $i_{r_{m-1}}(P_m)+c_m > -k_m$. By the index theorem there is
 $p\in P_m\setminus \{r_{m-1}, r_m, \zeta_m\}$ such that $i_p(P_m)\neq 0$. Thus
 $\tilde\xi$ can be extended through $p$ to a neighborhood of $r_m$ and $r_{m-1}$ and from there
 to a neighborhood of $P_1\cup...\cup P_m\setminus \zeta$.
 \end{proof}

 \begin{Lemma}
 Suppose that in the linear chain $\mathcal C(r_m)$, $P_1$ is dicritical and  the $P_l$, $l\neq 1$,
 are nondicritical. Then either $\tilde\xi$ extends to
 $P_1\cup...\cup P_m\setminus \zeta$, or
 the chain $\mathcal C_2(r_m)=(P_l)_{l=2}^m$ is minimal.
 \end{Lemma}

 \begin {proof}
 This is clearly a consequence of the Proposition as $i_{r_1}(P_2)=0$, and $\tilde\xi$ extends to
 $P_1\setminus\{r_1\}$.

 \end {proof}

\begin{proof}[Proof of Theorem~\ref{Theorem:Globalization}]
Let us now prove Theorem~\ref{Theorem:Globalization}.  Suppose $(U_n, \pi_n, D_n, \fa_n)$
is the resolution of $\fa$ at $0\in\bc^2$. Consider a
 linear chain $\mathcal C(r)=(P_l)_{l=1}^m$ of order $k$ with origin at a point
 $r\in D_{n-1}$. In this case we can take $\zeta=\emptyset$
 in Proposition~\ref{Proposition:locallyextensible} and assume that
 $\mathcal C(r)$ is minimal. From (ii) in the definition of minimal chain
 we obtain that $i_r(P)=-[k_m,...,k_h]\geq -[k_m,...,k_1]$. Since $k=[k_m,...,k_1]$ is
 the number of times that the point $r$ was blown-up to create $\mathcal C(r)$, we obtain
 that the index of $r$ before the creation of $\mathcal C(r)$ is $i_r(P)_b=i_r(P)-k\geq 0$.
 Thus, for the linear chains in $D_{n-1}\setminus D_{n-2}$ the origins of
 the linear chains in $D_{n}\setminus D_{n-1}$ contribute with a positive index.
 Thus we can apply again the Proposition~\ref{Proposition:locallyextensible}.

 Finally, we consider $\fa_0$. Let $\mathcal C(q_1),..., \mathcal C(q_t)$ be all linear
 chains starting at the reduced singularities in $D_0$. Let $l_i$ denote the order of
 $\mathcal C(q_i)$. Then $i_{q_i}(D_0)\geq 0$, for $i=1,...,t$. Since the self-intersection
 number of $D_0$ is $-1$, there must exist a point $s\in D_0\setminus \{q_1,..., q_t\}$
 such that $i_s(D_0)\neq 0$. Therefore $\xi$ can be extended to a neighborhood of
 $ D_0\setminus \{q_1,..., q_t\}$ and from there to the whole of $D_1\cup...\cup D_n$.
\end{proof}

\section{Foliations with  projective transverse structure}
\label{section:Introduction}
The  Globalization Theorem has some important consequences in the
study of transverse structure of holomorphic foliations with singularities.
We focus on the case of projective transverse structures, which is the
general case in codimension one (the affine and additive remaining cases
are viewed as subcases).

\subsection{Transversely projective foliations with singularities}

Let $\fa$ be a codimension one holomorphic foliation on a
connected complex manifold $M^m$, of dimension $m \ge 2$, having
singular set $\sing(\fa)$ of codimension $\ge 2$. The foliation
$\fa$ is {\it transversely projective\/} if there is an open cover
$\{ U_j ,\, j\in J \}$ of $M\backslash\sing(\fa)$ such that on
each $U_j$ the foliation is given by a holomorphic submersion
$f_j\colon U_j \to \ov\bc$ and on each intersection
$U_i \cap U_j \ne \emptyset$
we have
$f_i = \frac{a_{ij}f_j+b_{ij}}{c_{ij}f_j+d_{ij}}$
for some locally constant functions
$a_{ij}, b_{ij}, c_{ij}, d_{ij}$
with
$a_{ij}d_{ij} - b_{ij}c_{ij} = 1$.
If we have
$f_i = a_{ij}f_j + b_{ij}$
for locally constant functions
$a_{ij} \ne 0$, $b_{ij}$
then
$\fa$ is {\it transversely affine\/} in $M$ (\cite{Scardua1}).
In few words,  a holomorphic foliation $\fa$ of codimension one and
having singular set $\sing(\fa)$ of codimension $\ge 2$ in a complex
manifold $M^m$,  $m \ge 2$ is {\it transversely projective\/}
if the underlying nonsingular foliation
$ \fa|_{M\backslash\sing(\fa)}$
is transversely projective on
$ M\backslash\sing(\fa)$.
 Basic references for transversely affine and
transversely projective foliations are found in \cite{Godbillon}.

\begin{Remark} {\rm Assume that the dimension is $m=2$. Let $q\in\sing(\fa)$ be an
isolated singular point and $U$ a small bidisc such that
$\sing(\fa) \cap U =\{q\}$. Then $U\setminus \{q\}$ is
simply-connected and therefore
$\fa\big|_{U\setminus \{q\}}$ is given by a holomorphic
submersion $f\colon U\setminus \{q\}\to \ov \bc$ (\cite{Scardua1}). By Hartogs'
classical extension theorem \cite{Gunning} the map $f$ extends as
a meromorphic function $f \colon U \to \ov \bc$ (possibly with
an indeterminacy point at $q$). Thus, according to our definition, the
singularities of a foliation admitting a projective transverse
structure are all of type $df=0$ for some local meromorphic
function. For example we can consider $\fa$ given in a
neighborhood of the origin  $0 \in\bc^2$ by $ k xdy - \ell ydx=0$
where $k, \ell \in \mathbb N$. Then $\fa$ is transversely
projective in this neighborhood and given by the meromorphic
function $f = y^k / x^\ell$. Nevertheless, in this work we will be
considering foliations which are transversely projective in the
complement of {\it codimension one invariant divisors}. Such divisors
may, a priori, admit other types of singularities. In particular,
they can exhibit singularities which do not admit meromorphic
first integrals. An example is given by a {\it hyperbolic
singularity} of the form $xdy - \lambda ydx=0$ where $\lambda \in
\bc \setminus \mathbb R$. The corresponding foliation is
transversely projective (indeed, transversely affine) in the
complement of the set of separatrices $\{x=0\}\cup\{y=0\}$.
However, an easy computation with Laurent series shows that the
foliation admits no meromorphic first integral in a neighborhood
of the origin minus the two coordinate axes.}
\end{Remark}

\subsection{Projective transverse structures and differential forms}
\label{subsection:Projectivetriples}

Let $\fa$ be  a codimension one holomorphic  foliation with
singular set $\sing(\fa)$ of codimension $\ge 2$ on a complex
manifold $N$. The existence of a projective transverse structure
for $\fa$ is equivalent to the existence of suitable triples of
differential forms as follows:

\begin{Proposition} [\cite{Scardua1}]
\label{Proposition:forms}  Assume that $\fa$ is given by an
integrable holomorphic one-form $\Om$ on $N$ and suppose that
there exists a holomorphic one-form $\eta$ on $N$ such that
$\text{\rm(P1) } d\Om = \eta \wedge \Om$. Then $\fa$ is
transversely projective on $N$ if and only if there exists a
holomorphic one-form $\xi$ on $N$ such that $\text{\rm(P2) } d\eta
= \Om \wedge \xi$ and $\text{\rm(P3) }d\xi = \xi \wedge \eta$.
\end{Proposition}
\smallskip

\noindent This motivates the following definition:

\begin{Definition}{\rm Given holomorphic  one-forms (respectively,
meromorphic one-forms) $\Om$, $\eta$ and $\xi$ on $N$ we shall say
that $(\Om,\eta, \xi)$ is a {\it holomorphic projective triple\/}
(respectively, a {\it meromorphic projective triple\/}) if they
satisfy relations (P1), (P2) and (P3) above. }
\end{Definition}

With this notion Proposition~\ref{Proposition:forms} says that
$\fa$ is transversely projective on $N$ if and only if the
holomorphic  pair ($\Om$, $\eta$) may be completed to a
holomorphic projective triple. If for a holomorphic projective
triple we have $d\eta=0$ and $\xi=0$ then the projective
transverse structure is indeed an affine transverse structure (cf.
\cite{Scardua1}). Also according to \cite{Scardua1} we may  perform
modifications in a holomorphic or meromorphic projective triple as follows:

\begin{Proposition}
\label{Proposition:modificationforms}
\begin{itemize}
\item[\rm(i)] Given a meromorphic projective triple $(\Om, \eta,
\xi)$ and meromorphic functions $g$, $h$ on $N$ we can define a
meromorphic projective triple as follows:

{\rm (M1)}\,\, $\Om' = g\,\Om$

{\rm (M2)}\,\, $\eta' = \eta + \frac{dg}{g} + h\,\Om$

{\rm (M3)}\,\, $\xi' = \frac 1g\,\big(\xi - dh - h\eta -
\frac{h^2}{2}\,\Om\big)$

\item[\rm(ii)] Two holomorphic projective triples $(\Om,\eta,\xi)$
and $(\Om', \eta', \xi')$ define the same projective transverse
structure for a given foliation $\fa$ if and only if  we have {\rm
(M1), (M2)} and {\rm (M3)} for some holomorphic functions $g$, $h$
with $g$ nonvanishing.
\end{itemize}

\end{Proposition}

\smallskip

\noindent This last proposition   implies that  suitable meromorphic
projective triples also define projective transverse structures.

\begin{Definition} {\rm A meromorphic projective triple
$(\Omega ', \eta ', \xi ')$
is {\it true\/} if it can be written locally as in  (M1), (M2) and
(M3) for some (locally defined) holomorphic projective triple
$(\Om, \eta, \xi)$ and some (locally defined) meromorphic
functions $g, h$.}
\end{Definition}

\smallskip

\noindent As an immediate consequence we obtain:

\begin{Proposition}
\label{Proposition:true} A true projective triple $(\Om',
\eta', \xi')$ defines a transversely projective foliation $\fa$
given by $\Om'$ on $N$.
\end{Proposition}

\smallskip

\noindent The uniqueness of a meromorphic projective triple is
described by the  following lemma from \cite{Scardua1}:

\begin{Lemma}
\label{Lemma:ell} Let $(\Om,\eta,\xi)$ and $(\Om, \eta, \xi')$ be
meromorphic projective triples. Then $\xi' = \xi +F\,\Om$ for some
meromorphic function $F$ in $N$ with $d\,\Om = -\frac 12\,
\frac{dF}{F} \wedge \Om$.
\end{Lemma}

\noindent We can rewrite the condition on $F$ as $d(\sqrt {F}
\,\Om) = 0$. This implies that if the projective triples $(\Omega,
\eta, \xi)$ and $(\Omega, \eta, \xi ^\prime)$ are not identical
then the foliation defined by $\Omega$ is transversely affine
outside the codimension one analytical invariant subset $\Lambda:=(F=0)\cup (F=\infty)$.
(\cite{Scardua1}).

\subsection{Solvable groups of local diffeomorphisms}
We state a well-known technical result.
\begin{Lemma}
\label{Lemma:linearization} Let $G < \Diff(\bc,0)$ be a solvable
subgroup of germs of holomorphic diffeomorphisms fixing the origin
$0 \in \bc$.
\begin{itemize}
\item[\rm(i)] If $G$ is nonabelian and the group of commutators
$[G,G]$ is not cyclic then $G$ is analytically conjugate to a
subgroup of $\bh_k = \big\{z \mapsto
\frac{az}{\sqrt[k]{1+bz^k}}\big\}$ for some $k \in \bn$.
\item[\rm(ii)] If $f \in G$ is of the form $f(z) = e^{2\pi
i\la}\,z +\dots$ with $\la \in\bc\backslash \bq$ then $f$ is
analytically linearizable in a coordinate that also embeds $G$ in
$\mathbb H_k$.
\end{itemize}

\end{Lemma}

\begin{proof} (i) is in \cite{Cerveau-Moussu}. Given
$f \in G$ as in (ii) then by (i) we can write $f(z) =
\frac{e^{2\pi i\la}\,z}{\sqrt[k]{1+bz^k}}$ for some $k \in \bn$,
\, $b \in \bc$. Since $\la \in \bc\backslash\bq$ the homography
$H(z) = \frac{e^{2\pi i\la}\,z}{1+bz}$ is conjugate by another
homography to its linear part $z \mapsto e^{2\pi i\la}\,z$ and
therefore $f$ is analytically linearizable.
\end{proof}

\vglue .3in

\subsection{Extension to irreducible singularities}

\label{section:extensionirreducible}

Throughout this section $\fa$ will denote a holomorphic foliation
induced by a holomorphic one-form $\Omega$ defined on
a neighborhood of the origin $0 \in \bc^2$ and such
that $\sing(\fa) = \{0\}\in \bc^2$. Denote by $\sep(\fa,0)$ the
germ of all the separatrices of $\fa$ through $0\in\bc^2$.
We assume that the origin is an irreducible singularity. This
means that in suitable local coordinates $(x,y)$ in a neighborhood
$U$ of the origin, we have local normal forms for the restriction
$\fa\big|_{U}$  given by (Irr.1) or (Irr.2).


\begin{Lemma}[nonresonant {\it linearizable} case]
\label{Lemma:linearizable} Suppose that
$\Om= g(xdy-\la ydx)$
for some holomorphic nonvanishing function $g$ in $U$, and
$\lambda \in \bc \setminus \mathbb Q_+$. Let $F$ be a meromorphic function in
$U^*=U\setminus \{xy=0\}$, such that
$d\Om = -\frac 12\,\frac{dF}{F}\wedge \Om$.
If $\lambda \notin\mathbb Q$ then $F$ extends
 to $U$ as a meromorphic function,  $F = \tilde c.(gxy)^{-2}$ for some
constant $\tilde c \in \bc$.
\end{Lemma}

\begin{proof}
First we remark that from equation $d\Om = -\frac
12\,\frac{dF}{F}\wedge \Om$ we have that the set of poles of $F$ is
invariant by $\Omega$. Therefore, since the only separatrices of the
form $\Omega$ in $U$ are the coordinates axes, we can assume that
$F$ is holomorphic in $U^*$.
 Fix a complex number
$a\in \bc$ and introduce the one-form $\eta_0 = \frac{d(xyg)}{xyg}
+ a(\frac{dy}{y} - \lambda \frac{dx}{x})$ in $U$. Since
$d(\frac{\Omega}{gxy})= \frac{dy}{y} - \lambda \frac{dx}{x}$ is
closed it follows that $d\Omega= \eta_0 \wedge \Omega$. Thus the
one-form $\Theta:= -\frac 12\,\frac{dF}{F} -  \eta_0$ is closed
meromorphic in $U^*$ and satisfies $\Theta \wedge \Omega = d
\Omega - d \Omega = 0$. This implies  that $\Theta \wedge
(\frac{dy}{y} - \lambda \frac{dx}{x}) =0$ in $U^*$ and therefore
we have $\Theta =h. (\frac{dy}{y} - \lambda \frac{dx}{x})$ for
some meromorphic function $h$ in $U$. Taking exterior derivatives
we conclude that $dh\wedge( \frac{dy}{y} - \lambda
\frac{dx}{x})=0$ in $U^*$ and therefore $h$ is a meromorphic first
integral for $\Omega$ in $U^*$. Since $\lambda \notin \mathbb Q$
we must have $h=c$, a constant: indeed, write $h= \sum\limits_{i,j
\in\mathbb Z} h_{ij} x^iy^j$ in Laurent series in a small bidisc
around the origin. Then from $dh\wedge( \frac{dy}{y} - \lambda
\frac{dx}{x})=0$ we obtain $(i + \lambda j)h_{ij}=0, \, \forall
(i,j) \in \mathbb Z\times \mathbb Z$ and since $\lambda \notin
\mathbb Q$ this implies that $h_{ij}=0, \, \forall (i,j) \ne
0 \in \bc^2$.

This already shows that the one-form $\Theta$ always extends as a
meromorphic one-form with simple poles to $U$ and therefore the
function $F$ extends as a  meromorphic function to $U$. The
residue of $\Theta$ along the axis $\{y=0\}$ is given by
$\Res_{\{y=0\}} \Theta = -\Res_{\{y=0\}} \frac{1}{2}\frac{dF}{F}
 - \Res_{\{y=0\}} \eta_0= - \frac{1}{2} k - (1+a)$ where
 $k\in \mathbb N$ is the order of
$\{y=0\}$ as a set of zeroes of $F$ or minus the order of
$\{y=0\}$ as a set of poles of $F$. Thus by a suitable choice of
$a$ we can assume that $\Res_{\{y=0\}} \Theta=0$ and therefore by
the expression $\Theta =c (\frac{dy}{y} - \lambda \frac{dx}{x})$
we conclude that, for such a choice of $a$, we have $0=\Theta=
-\frac 12\,\frac{dF}{F} -  \eta_0$ and thus $ -\frac
12\,\frac{dF}{F} =  \frac{dx}{x} + \frac{dy}{y} + \frac{dg}{g} +
a(\frac{dy}{y} - \lambda\frac{dx}{x})$ and therefore, comparing
residues along the axes $\{y=0\}$ and $\{x=0\}$ we obtain that
$1+a\in \mathbb Q$ and $ 1 - a\lambda \in \mathbb Q$. Since
$\lambda \notin\mathbb Q$ the only possibility is $a=0$. This
proves that indeed $-\frac 12\,\frac{dF}{F} = \frac{dx}{x} +
\frac{dy}{y} + \frac{dg}{g}$ in $U$ and integrating this last
expression we obtain  $F = \tilde c (gxy)^{-2}$ for some constant
$\tilde c \in \bc$. This proves the lemma. \vglue.2in
\end{proof}

\begin{Remark}
\label{Remark:Touzet} {\rm {\bf (i)} According to \cite{Touzet},
Theorem II.3.1, {\it a nondegenerate nonresonant singularity $xdy -
\lambda y dx + \Omega_2(x,y)=0, \, \lambda \in \mathbb C \setminus
\mathbb Q_+$,  is analytically linearizable if and only if the
corresponding foliation $\fa$ is transversely projective in
$U\setminus \sep(\fa,U)$ for some neighborhood $U$ of the
singularity}. {\bf (ii)}  Let now $\fa$ be of resonant type or of
saddle-node type.  According to \cite{Touzet}, Theorem II.4.2, {\it
the foliation admits a meromorphic projective triple near the
singularity if and only if on a neighborhood of $0\in \mathbb C^2$,
$\fa$ is the pull-back of a Riccati foliation on $\ov \bc \times \ov
\bc$ by a meromorphic map}. The proof of this theorem is based on
the study and classification of the Martinet-Ramis cocycles of the
singularity expressed in terms of some classifying holonomy map of a
separatrix of the singularity. For a resonant singularity any
 of the two separatrices has a {\it classifying holonomy} ({\it i.e.},
 the analytical conjugacy class of the singularity germ is determined by
 the analytical conjugacy class of the holonomy map of the separatrix) and for a
saddle-node it is necessary to consider the strong separatrix
holonomy map. Thus we conclude that the proof given in
\cite{Touzet}  works if we only assume the existence of a
meromorphic projective triple $(\Om', \eta', \xi')$ on a
neighborhood $U_0$ of $ \Lambda\setminus (0,0)$, where
$\Lambda \subset \sep(\fa,U)$ is any separatrix in the resonant case,
and the strong separatrix if the origin is a saddle-node.

 }
\end{Remark}

\begin{Lemma}[general nonresonant case]
\label{Lemma:noresonantextends}  Suppose that the origin is a
nondegenerate nonresonant singularity of the foliation $\fa$. Assume
that $\fa$ is transversely projective on $U\setminus \sep(\fa,U)$.
Let $\eta$ be a meromorphic one-form on $U$ and $\xi$ be a
meromorphic one-form on $U\setminus \sep(\fa,U)$ such that on $U
\setminus \sep(\fa,U)$ the one-forms $\Omega, \eta, \xi$ define a
true projective triple. Then $\xi$ extends as a meromorphic one-form
to $U$.
\end{Lemma}

\begin{proof} By hypothesis  the foliation is given in suitable local
coordinates around the origin by $xdy - \la ydx + \omega_2(x,y) = 0$
where $\la \in \bc\backslash\bq$\,, \, $\omega_2(x,y)$ is a
holomorphic one-form of order $\ge 2$ at $0\in \bc^2$.
\begin{Claim} The singularity is analytically linearizable.
\end{Claim}

Indeed, if  $\la \notin \re_-$ then the singularity is in the
Poincar\'e domain with no resonance and  by
Poincar\'e-Linearization Theorem the singularity is analytically
linearizable.  Assume now that $\la \in \re_-\backslash\bq_-$\,.
In this case the singularity is in the Siegel domain and, a
priori, it is not clear that the singularity is linearizable.
Nevertheless, by hypothesis  $\fa$ is transversely projective in
$U^* = U\setminus \sep(\fa,U)$ and  by Remark~\ref{Remark:Touzet}
(i) the singularity $0\in \bc^2$ is analytically linearizable.
This proves the claim.

\vglue.1in

Therefore we can suppose that  $\Om\big|_U = g(xdy-\la ydx)$ for
some holomorphic nonvanishing function $g$ in $U$.  We define
$\eta_0 = \frac{dg}{g} + \frac{dx}{x}  + \frac{dy}{y}$ in $U$.
Then $\eta_0$ is meromorphic and satisfies $d\Om = \eta_0 \wedge
\Om$ so that $\eta = \eta_0 + h\Om$ for some meromorphic function
$h$ in $U$. We also take $\xi_0 = 0$ so that $d\eta_0 = 0 = \Om
\wedge \xi_0$ and $d\xi_0 = 0 = \xi_0 \wedge \eta$. The triple
$(\Omega,\eta_0,\xi_0)$ is a meromorphic projective triple in $U$
so that according to
Proposition~\ref{Proposition:modificationforms} we can define a
meromorphic projective triple $(\Om,\eta,\xi_1)$ in $U$  by
setting $\xi_1 = \xi_0 - dh - h\eta_0 - \frac{h^2}{2}\,\Om = -dh -
h\eta_0 - \frac{h^2}{2}\,\Om$. Then we have by
Lemma~\ref{Lemma:ell} that $\xi = \xi_1 + \ell.\Om$ for some
meromorphic function $\ell$ in $U^*$ such that $d\Om = -\frac
12\,\frac{d\ell}{\ell} \wedge \Om$.

By Lemma~\ref{Lemma:linearizable} above we have $\ell =\tilde c
.(gxy)^{-2}$ in $U^*$ and therefore $\xi$ extends to $U$ as $\xi =
\xi_1 + \tilde c. (gxy)^{-2}$ in $U^*$. This proves the lemma.
\end{proof}

\subsection{Extension from a separatrix of an irreducible singularity}

\begin{Lemma}
\label{Lemma:ellextension} Let $\Omega$ be a holomorphic one-form of type
{\rm(Irr.1)} or {\rm(Irr.2)}
defined on $U$. Let $ S\subset \sep(\fa,U)$  be a separatrix of
$\fa\big|_{U}$  which is  a strong  manifold of $\fa$, in case $0\in
\bc^2$ is a saddle-node. Let $F$ be a meromorphic function in $U$
minus the other separatrix of $\fa$ in $U$ such that $d\,\Om =
-\frac 12\, \frac{dF}{F} \wedge \Om$. Then $F$ extends as a
meromorphic function to $U$; indeed we have the following
possibilities for $\Omega$ and $F$ in suitable coordinates in a
neighborhood of the origin:
\begin{itemize}

\item[{\rm (i)}] $\Omega=g(xdy - \lambda ydx)$ for some $\lambda
\in \bc\setminus \{0\}$ and some meromorphic function $g$. If
$\{\lambda, \frac{1}{\lambda}\}\cap \mathbb N= \emptyset$ then
$F=\tilde c. (gxy)^{-2}$ for some constant $\tilde c$. If $\la =
-\frac{k}{\ell}\in \mathbb Q_-$ where $k,\ell \in \mathbb N, \,
<k,\ell>=1$ then $F=\tilde c (gxy)^{-2} . \,\vr(x^k y^\ell)$ for
some constant $\tilde c\in \bc$ and some meromorphic function
$\vr(z)$ in a neighborhood of the origin $0 \in \bc$.

\item[{\rm (ii)}] $\Omega=g.\Om_{1,\ell} = g( y\,dx + \ell x(1 +
\frac{\sqrt{-1}}{2\pi}\, xy^\ell)dy)$ where $\ell\in \mathbb N$
and $g$ is meromorphic. We have $F=\tilde c. (gx^{2}
y^{\ell+1})^{-2}$ for some constant $\tilde c$.

\item[{\rm (iii)}] $\Omega=g.\Omega_{(2)}=g(xdy - y^2dx)$ for some $g$
meromorphic. We have $F=\tilde c. (gxy^2)^{-2}$ for some constant
$\tilde c$.
\end{itemize}

In all cases $ S$ is given by $\{y=0\}$ and the function $F$
extends as meromorphic function to a neighborhood of the origin.
\end{Lemma}

\begin{proof}
 We define the one-form  $\eta=-\frac 12\,
\frac{dF}{F}$. Then $\eta$ is a closed meromorphic one-form in
$U\setminus [\sep(\fa,U)\setminus  S]$ such that $d\Omega=
\eta\wedge \Omega$, moreover the polar set of $\eta$ is contained in
$ S$ and has order at most one. If $\eta$ is holomorphic in
$U\setminus [\sep(\fa,U)\setminus  S]$ then the foliation $\fa$ is
transversely affine there and therefore the holonomy map $h$ of the
leaf $L_0=  S\setminus\{0\}$ is linearizable. Since the origin is
irreducible and $ S$ is not a central manifold, the conjugacy class
of this holonomy map classifies the foliation up to analytic
conjugacy (\cite{Martinet-Ramis1},\cite{Martinet-Ramis2}). Thus
the singularity is itself linearizable. Assume now that
$(\eta)_\infty \ne \emptyset.$ In this case we have the residue of
$\eta$ along $ S$ given by $\Res_{ S}\eta= -\frac12 \, k$ where $k$
is either the order of $ S$ as zero of $F$ or minus the order of $
S$ as pole of $F$. We have two possibilities:

(a) If $-\frac12 \, k \notin \{2,3,...\}$ then according to
\cite{Scardua1} Lemma 3.1 the holonomy map of the leaf $L_0$ is
linearizable and as above the singularity itself is linearizable.

(b) If  $-\frac12 \, k = t +1\ge 2$ for some $t\in \mathbb N$ then
by \cite{Scardua1} Lemma 3.1  there is a conjucacy between the holonomy
map of $L_0$ and a map of the form $h(z)= \frac{\alpha z}{(1+ \beta
z^t)^{\frac{1}{t}}}$, i.e., this is a finite ramified covering of
an homography.

\noindent{\bf 1st. case}: Suppose that the singularity is
nondegenerate, say $\Omega = xdy - \lambda ydx +...$.  If $\{\la,
\frac{1}{\lambda}\} \cap \mathbb N = \emptyset$ then  $a = h'(0) =
e^{2\pi i/\la}\ne 1$ and by Lemma~\ref{Lemma:linearization} (ii) the
holonomy map $h$ is analytically linearizable. Therefore, as
remarked above, in this case the singularity $q_{j_0} \in
\sing(\fa)$ is analytically linearizable. Thus we can assume that
$\lambda = -\frac{1}{\ell}$ for some $\ell\in \bn$. In  this case,
either the holonomy is the identity (and therefore linearizable) or
there is an analytical conjugacy to the corresponding holonomy of the
separatrix $(y=0)$ for the germ of  a singularity $ \Om_{k,\ell} = k
y\,dx + \ell x(1 + \frac{\sqrt{-1}}{2\pi}\, x^ky^\ell)dy$ for $k=1$;
such a singularity is called a {\it nonlinearizable resonant
saddle}. Therefore, by \cite{Mattei-Moussu} and
\cite{Martinet-Ramis1} we may assume that $\fa|_U$ is of the form
$\Om_{1,\ell} = 0$ in the variables $(x,y) \in U$.

\vglue.1in

\noindent{\bf 2nd. case}:  Now we consider the case for which the
singularity is a saddle-node. By hypothesis, $ S$ is the strong
manifold of the saddle-node and therefore its holonomy $h$ is
tangent to the identity and thus it is analytically conjugated to
$z\to \frac{z}{1+z}$ which is conjugated to the corresponding
holonomy map of the separatrix $(y=0)$ for the saddle-node
$\Omega_{(2)} = y^{2} dx - x dy$
so that by \cite{Martinet-Ramis2} the foliation $\fa$ is analytically
conjugated to $\Omega_{(2)}$ in a neighborhood of the origin.

\vglue.1in So far we have proved that the singularity is either
analytically linearizable,  analytically conjugated to
$\Omega_{1,\ell}=0$ if  it is resonant and not analytically
linearizable, or analytically
 conjugated to $\Omega_{(2)}=0$ if it is a saddle-node. We shall now
 work with these three models in order to conclude the extension
 of $F$ to $U$.

\noindent (i) In the linearizable case we can write $ S: \{y=0\}$
and  $\Omega= g(xdy - \lambda ydx)$
 for some holomorphic nonvanishing function $g$ in $U$.
 If $\lambda \notin \mathbb Q$
 then by Lemma~\ref{Lemma:linearizable} $F$ extends meromorphicaly to $U$.
 Assume now that $\lambda = -\frac{1}{\ell} \in \mathbb
 Q_-$.  Recall that $\eta=-\frac{1}{2}\frac{dF}{F}$
 satisfies $d\Omega= \eta \wedge \Omega$ and $d\eta=0$. If we
 introduce $\tilde\eta_0 = \frac{d(gxy)}{gxy}$ then we have $d\Omega=
 \tilde\eta_0 \wedge \Omega$ and therefore $(\eta - \tilde\eta_0 )
 \wedge\Omega=0$ so that $(\eta - \tilde\eta_0 )\wedge
 (\frac{dy}{y} - \lambda \frac{dx}{x})=0$ and
 then $\eta= \tilde\eta_0 + H.(\frac{dy}{y} - \lambda \frac{dx}{x}) $ for some
 meromorphic function $H$ in $U_0:=U\setminus\{y=0\}$. Since $\eta$ and
 $\tilde\eta_0$ are closed we conclude that
 $d(H.(\frac{dy}{y} - \lambda \frac{dx}{x}))=0$ in $U_0$.
 Write now $H=\sum\limits_{i,j \in \mathbb Z} H_{ij} x^i y ^j$ in
 Laurent series in a small bidisc around the origin. We obtain from
 the last equation that $(i + \lambda j) H_{ij}=0, \, \forall i,j \in \mathbb
 Z$ (for $\lambda \notin \mathbb Q$ this implies, again, that $H=H_{00}$ is
 constant).  Thus we  have $\Omega \wedge d(xy^\ell)=0$ and also
 $F=\vr(xy^\ell)$ for some function $\vr(z)=\sum\limits_{t\in
 \mathbb Z} \vr_t z^t$ defined in a punctured disc around the
 origin. Nevertheless, the function $F$ is meromorphic along the
 axis $\{y=0\}$ and therefore $\vr$ extends  to the
 origin $0 \in \bc$ as a meromorphic function and thus $F$ extends  to a neighborhood of the origin
 as  $F=\vr(xy^\ell)$.

\vglue.1in

(ii) In the nondegenerate nonlinearizable case we can write $ S :
\{y=0\}$ and $\Omega = g\, \Om_{1,\ell} =g( y\,dx + \ell x(1 +
\frac{\sqrt{-1}}{2\pi}\, xy^\ell)dy)$ for some holomorphic
nonvanishing function $g$ on $U$. Define $\tilde\eta_0 =
\frac{d(gx^{2} y^{\ell+1})}{gx^{2} y^{\ell+1}}$. As above we
conclude that $\eta = \tilde\eta_0 + H. (n \frac{dx}{x^{2}y^\ell}
+ m \frac{dy}{x y^{\ell+1}} + \frac{m \sqrt{-1}}{2 \pi}
\frac{dy}{y})$ for some meromorphic function $H$ in $U_0$ such
that $dH\wedge (n \frac{dx}{x^{2}y^\ell} + \ell \frac{dy}{x
y^{\ell+1}} + \frac{\ell \sqrt{-1}}{2 \pi} \frac{dy}{y})=0$. In
other words, $H$ is a meromorphic first integral in $U_0$ for the
foliation $\fa$. This implies that $H$ is constant. In order to
see this it is enough to use Laurent series as above.
Alternatively one can argue as follows. If $H$ is not constant
then the holonomy map $h$ of the leaf $L_0\subset  S$ leaves
invariant a nonconstant meromorphic map (the restriction of the
first integral $H$ to a small transverse disc to $ S$). This
implies that $h$ is a map with finite orbits and indeed $h$ is
periodic. Nevertheless this is never the case of the holonomy map
of the separatrix $\{y=0\}$ of the foliation $\Omega_{1,\ell}$.
Thus the only possibility is that $H$ is constant.

\vglue.1in

(iii) In the saddle-node case we can write $\Omega=g\, \Omega_{(2)}=
g(xdy - y^{2}dx)$ for some holomorphic nonvanishing function $g$
in $U$. Defining $\tilde\eta_0 = \frac{d(gxy^{2})}{gxy^{2}}$ and
proceeding as above we conclude that $\eta= \tilde\eta_0 +
H.(\frac{dy}{y^{2}} - \frac{dx}{x})$ for some meromorphic function
$H$ in $U_0= U \setminus \{x=0\}$ such that $dH\wedge
(\frac{dy}{y^{2}} - \frac{dx}{x})=0$, i.e., $H$ is a meromorphic
first integral for the saddle-node in $U_0$. A similar
argumentation as above, either with Laurent series or with
holonomy arguments, shows that $H$ must be constant.

\vglue.2in We have therefore proved that in all cases $\eta=
\tilde\eta_0 + H.\omega$ for some meromorphic function $H$ in $U$
and some meromorphic closed one-form $\omega$ in $U$. Moreover,
$H$ is constant except in the resonant case. This shows that
$\eta=-\frac12\frac{dF}{F}$ extends  to $U$ as a meromorphic one-form and
therefore also $F$ extends  to $U$ as a meromorphic function, the lemma is
proved.
\end{proof}

\begin{Lemma}
\label{Lemma:xiextends}  Fix a separatrix $\Lambda \subset
\sep(\fa,U)$ which is not a central manifold, in case the origin
is a saddle-node. Let $\eta$ be a meromorphic one-form in $U$ and
$\xi$ be a meromorphic one-form in $U \setminus
[\sep(\fa,U)\setminus \Lambda]$ such that in $U \setminus
\sep(\fa,U)$ the one-forms $\Omega, \eta, \xi$ define a projective
triple. Then  $\xi$ extends as a meromorphic one-form to  $U$.
\end{Lemma}

\begin{proof} The proof is based in the preceding results and in Theorem
II.4.2 of \cite{Touzet} (see Remark~\ref{Remark:Touzet}). Let us
analyze what occurs case by case:
\vglue.1in
\noindent{\bf Nondegenerate singularity}:
 First assume that $\fa$ is nondegenerate and
nonresonant.  By  Lemma~\ref{Lemma:noresonantextends} above the
singularity is analytically linearizable and the one-form $\xi$
extends  to $U$ as a meromorphic one-form.
Now we consider the resonant case, i.e.,  $\Omega= g(xdy - \lambda ydx +...)$
with $\la = - \frac{n}{m} \in \mathbb Q_-$ and that the
singularity is not analytically linearizable. As we have seen in
Remark~\ref{Remark:Touzet},  $\fa$ is the pull-back of a Riccati
foliation on $\ov\bc \times \ov\bc$ by some meromorphic map
$\sigma\colon U ---> \ov\bc \times \ov\bc$ provided that there is
a  meromorphic projective triple $(\Om', \eta', \xi')$ in a
neighborhood $W$ of a separatrix $\Lambda \subset \sep(\fa,U)$.
 From our hypothesis such a projective triple is
given by the restrictions of $\Omega$ and $\eta$ to $U\setminus
[\sep(\fa,U)\setminus \Lambda]$ and by the one-form $\xi$. Thus we
conclude that $\fa$ is a meromorphic pull-back of a Riccati
foliation and in particular there is a one-form $\xi'$ defined in
a neighborhood $\tilde U$ of the origin such that $(\Omega, \eta,
\xi')$ is a projective triple in this neighborhood. This implies
that $\xi= \xi' + \ell. \Omega$ in $\tilde U$ for some meromorphic
function $\ell$ in $\tilde U$ such that $d\Omega = - \frac12
\frac{d\ell}{\ell}$ in $\tilde U$. Now we have two possibilities.
Either $\xi=\xi'$ in $\tilde U$ or $\ell \not \equiv 0$. In the
first case $\xi$ extends  to $U$ as a meromorphic one-form,  $\xi=\xi'$. In
the second case we apply Lemma~\ref{Lemma:ellextension} above in
order to conclude that the singularity is analytically
normalizable and $\ell$ extends as a meromorphic function to $U$.
Finally,  suppose  the singularity is resonant analytically
linearizable, that means  $\fa$ is given in $U$ by $\Om = g\big(xdy +
\frac{n}{m}\,ydx\big)$ where $n,m \in \bn$ and $g$ is a
meromorphic function in $U$. In this case as above we define
$\eta_0 = \frac{dg}{g} + \frac{dx}{x} + \frac{dy}{y}$\,, write
$\eta = \eta_0 + h\Om$ and define $\xi_0=0$, \, $\xi_1 =
\xi_0-dh-h\eta_0 - \frac{h^2}{2}\,\Om = -dh - h\eta_0 -
\frac{h^2}{2}\,\Om$. Now we have $\xi = \xi_1 + \ell\Om$ for some
meromorphic function $\ell$ in $U^*$. In this case we have from
$d\ell = -\frac12\,\frac{d\ell}{\ell} \wedge \Om$ that
$\ell(gxy)^2 = [\vr(x^ny^m)]^2$ for some meromorphic function
$\vr(z)$ defined in a punctured neighborhood of the origin $0 \in
\bc$. In particular we conclude that since $\xi$ extends
 to some separatrix $\{x=0\}$ or $\{y=0\}$ as a meromorphic one-form then it
extends  to $U$ as a meromorphic one-form.

\vglue.2in
\noindent{\bf Saddle-node case}:
Finally, we assume that the origin is a saddle-node. We
write $\Om = g[y^{t+1}dx - (x(1+\la y^t)+\dots)dy]$ for some
holomorphic nonvanishing function $g$ in $U$.  Again by
Remark~\ref{Remark:Touzet}  there exists a meromorphic projective
triple $(\Om', \eta', \xi')$ for $\fa$ in $U$  which is given by a
meromorphic pull-back of a Riccati foliation projective triple. We
can assume that $\eta' = \eta$ and therefore $\xi = \xi ' + \ell
\Omega$ where $\ell$ is a meromorphic function in $U^*$ such that
$d\Omega= -\frac{1}{2} \frac{d\ell}{\ell}\wedge \Omega$. There are
two cases: If $\ell\equiv 0$ then $\xi$ extends as $\xi'$ to $U$.
Assume that $\ell\not \equiv 0$. In this case by
Lemma~\ref{Lemma:ellextension} the singularity is analytically
conjugated to $\Omega_{(t)}$ and the function $\ell$ extends
 to $U$ as a meromorphic function. Thus $\xi$ extends as a meromorphic
one-form to $U$. \end{proof}

\begin{Lemma}[noninvariant divisor]
\label{Lemma:noninvariantextension} Let be given a holomorphic
foliation $\fa$ on a complex manifold  $M$. Suppose that $\fa$ is
given by a meromorphic integrable one-form $\Omega$ which admits a
meromorphic one-form $\eta$ on $M$ such that $d \Omega= \eta \wedge
\Omega$. If $\fa$ is transversely projective in $M\setminus \Lambda$
for  some {\it noninvariant}  irreducible analytic subset $\Lambda
\subset M$ of codimension one then $\fa$ is transversely projective
in $M$.
\end{Lemma}
\begin{proof} Our argumentation is local, {\it i.e.}, we consider a small neighborhood $U$ of a generic point $q\in \Lambda$ where $\fa$ is transverse to $\Lambda$. Thus, since $\Lambda$ is not invariant by $\fa$, performing changes as $\Omega^\prime = g_1 \Omega$ and $\eta ^\prime = \eta + \frac{dg_1}{g_1}$  we can assume that $\Omega$ and $\eta$ have poles in general position with respect to $\Lambda$ in $U$. The existence of a projective transverse structure for $\fa$ off $\Lambda$ then gives a meromorphic one-form $\xi$ in $M \setminus \Lambda$ such $(\Omega, \eta, \xi)$ is a true projective triple in $M \setminus \Lambda$. For $U$ small enough we can assume that for suitable local coordinates $(x,y)=(x_1,...,x_n,y)\in U$ we have $\Lambda \cap U = \{x_1=0\}$ and also
\[
\Omega = gdy, \eta = \frac{dg}{g} + h dy
\]
for some holomorphic function $g, h \colon U \to \bc$ with $1/g$
also holomorphic in $U$. Then we have
\[
\xi=-\frac{1}{g}\big[dh + \frac{h^2}{2} dy\big]
\]
where
\[
d( \sqrt {\ell} g dy ) =0
\]
Thus, $\sqrt{\ell} g = \vr (y)$ for some meromorphic function
$\vr(y)$ defined for $x_1 \ne 0$ and therefore for $x_1=0$. This
shows that $\xi$ extends  to $W$ as a {\it holomorphic one-form} and then  the
projective structure extends to $U$. This shows that the transverse
structure extends to $\Lambda$.
\end{proof}

\noindent Summarizing the above discussion we obtain the following
proposition:

\begin{Proposition}
\label{Proposition:technical}
 Let $\mathcal F$ a holomorphic foliation in a neighborhood $U$ of the origin
 $0 \in \bc^2$ with an isolated  singularity at the origin. Suppose
 that $ \fa$ is transversely projective in $U\setminus \sep(\mathcal \fa,U)$ and let $(\Omega, \eta, \xi)$ be a
 meromorphic triple in $U\setminus \sep(\fa, U)$ with $\Omega$ holomorphic in $U$, $\eta$ meromorphic in $U$ and
 $\xi$ meromorphic in $U\setminus \sep(\fa,U)$.
Then the one-form $\xi$ is infinitesimally extensible with respect
to $\mathcal F$.
\end{Proposition}

From Proposition~\ref{Proposition:technical} and Theorem~\ref{Theorem:Globalization} we obtain:

\begin{Theorem}
\label{Theorem:projectiveextensible} Let $\mathcal F$ a holomorphic
foliation in a neighborhood $U$ of the origin
 $0 \in \bc^2$ with an isolated  singularity at the origin. Suppose
 that $ \fa$ is transversely projective in
 $U\setminus \sep(\mathcal \fa,U)$ and let $(\Omega, \eta, \xi)$ be a
 meromorphic triple in $U\setminus \sep(\fa, U)$ with $\Omega$ holomorphic
 in $U$, $\eta$ meromorphic in $U$ and
 $\xi$ meromorphic in $U\setminus \sep(\fa,U)$.
Then the one-form $\xi$ extends as a meromorphic one-form to a
neighborhood of the origin provided that it extends to some
fundamental domain of $\mathcal F$.
 \end{Theorem}

We recall that a germ of a foliation at the origin $0\in \bc^2$ is a
{\it generalized curve} if it  exhibits no saddle-node in its
resolution by blow-ups (\cite{C-LN-S}). The generalized curve is
{\it non-resonant} if each connected component of the invariant part
of the exceptional divisor contains some singularity of non-resonant
type. The inverse image of a fundamental domain of a non-resonant
generalized curve contains a fundamental of each singularity arising
in its resolution process. Therefore, from
Theorem~\ref{Theorem:projectiveextensible} we obtain:

\begin{Corollary}
\label{Corollary:projectiveextensible} Let $\mathcal F$ be a germ of
a non-resonant generalized curve at the origin $0\in \mathbb C^2$.
Suppose
 that $ \fa$ is transversely projective in
 $U\setminus \sep(\mathcal \fa,U)$ and let $(\Omega, \eta, \xi)$ be a
 meromorphic triple in $U\setminus \sep(\fa, U)$ with $\Omega$ holomorphic
 in $U$, $\eta$ meromorphic in $U$ and
 $\xi$ meromorphic in $U\setminus \sep(\fa,U)$.
Then the one-form $\xi$ extends to $U$  as a meromorphic one-form
 \end{Corollary}

We believe that Theorem~\ref{Theorem:Globalization} might have other
applications. For instance, consider two germs of holomorphic vector
fields with same set of separatrices and holomorphically equivalent
in a neighborhood of the singularity minus the local separatrices.
In this situation,  Theorem~\ref{Theorem:Globalization} may be an
useful tool in the investigation of  the existence of a holomorphic
equivalence for the germs in terms of their associated projective
holonomy groups.

\bibliographystyle{amsalpha}

\vglue.1in

\begin{tabular}{ll}
C\'esar Camacho  & \qquad  Bruno Sc\'ardua\\
IMPA-Estrada D. Castorina, 110 & \qquad Inst. Matem\'atica\\
Jardim Bot\^anico  & \qquad Universidade Federal do Rio de Janeiro\\
Rio de Janeiro - RJ  & \qquad  Caixa Postal 68530\\
CEP. 22460-320   & \qquad 21.945-970 Rio de Janeiro-RJ\\
BRAZIL &  \qquad BRAZIL
\end{tabular}

\end{document}